\documentclass{article}
\usepackage{graphicx} 
\usepackage{subfigure}
\usepackage{amsmath}
\usepackage{amsfonts}
\usepackage{amsthm}
\usepackage{mathrsfs}
\usepackage{amssymb}
\usepackage {color}
\usepackage{indentfirst}
\usepackage{float}
\graphicspath{{images/}} 
\usepackage{authblk}

\usepackage[all]{xy}
\usepackage{hyperref}
\hypersetup{hypertex=true,colorlinks=true,linkcolor=blue, anchorcolor=blue,citecolor=blue}
\usepackage[linkcolor=red,anchorcolor=blue, colorlinks=red]{adjustbox} 

\title{Perron's method and spherical ideal circle patterns with prescribed total geodesic curvatures}

\author{Lishan Li, Jun Hu, Yi Qi, Yu Sun}

\date{}

\begin{document}

\maketitle

\newtheorem{definition}{Definition}[section]
\newtheorem{theorem}[definition]{Theorem}
\newtheorem{proposition}[definition]{Proposition}
\newtheorem{lemma}[definition]{Lemma}
\newtheorem{remark}[definition]{Remark}
\newtheorem{corollary}[definition]{Corollary}

\begin{abstract}
In this paper, we apply the classical Perron method to give a proof of the existence and uniqueness/rigidity result of a circle pattern on a closed surface equipped with conical spherical metric when prescribed measures of the angles of intersecting circles stay in the range $(0, \frac{\pi}{2}]$ and total geodesic curvatures are assigned to the circles, which is recently obtained in \cite{MR4683863} via Colin de Verdi\`ere’s variation method. Then we show the convergence of Thurston's algorithm, which adjusts the geodesic curvatures of circles one by one based on the prescribed values for total geodesic curvatures of the circles, to the desired circle pattern in the setting of the result. 
\end{abstract}
\maketitle
\section{Introduction}
Let $S$ be a closed oriented surface equipped with a spherical, Euclidean or hyperbolic metric (possibly with singularities). By a {\em Delaunay circle pattern} on $S$ we mean a finite collection of open disks $\mathcal{D}=\{D_1,D_2.\cdots,D_n\}$ on $S$ satisfying the following two conditions:
\begin{itemize}
    \item  The complement of $\cup_{D\in \mathcal{D}} D$ in $S$ consists of finitely many points.

    \item  No disk of $\mathcal{D}$ has its boundary contained in the union of two other disks. 
\end{itemize}
One associates to $\mathcal{D}$ a weighted graph $G_{\mathcal{D}}$ on $S$ by defining its vertex set $V_{\mathcal{D}}$, edge set $E_{\mathcal{D}}$ and face set $F_{\mathcal{D}}$ as follows:
\begin{itemize}
    \item $V_{\mathcal{D}}=S\setminus \cup_{D\in \mathcal{D}} D$;

    \item Each face $f$ of $F_{\mathcal{D}}$ corresponds to a disk $D_f\in \mathcal{D}$;

    \item Each edge $e\in E_{\mathcal{D}}$ corresponds to a bigon where two disks of $\mathcal{D}$ intersect, and is weighted by the interior angle $\theta_e$ ($\in (0, \pi)$) of the bigon. 
\end{itemize}
Assume that the interior angles of all bigons related to the edges have measures in $(0, \frac{\pi}{2}]$. Then the bigons are disjoint from each other and none of them contains any singularities if there exist. Such circle patterns are called {\em ideal circle patterns}. 
Figure \ref{fig-2} illustrates a non-ideal circle pattern, which is determined by the graph on the right; Figure \ref{fig-5} shows an ideal circle pattern, which is determined by the graph on the right with weights on the edges in the range $(0, \frac{\pi}{2}]$.
\begin{figure}[H]
{
\centering
\includegraphics[height=4cm]{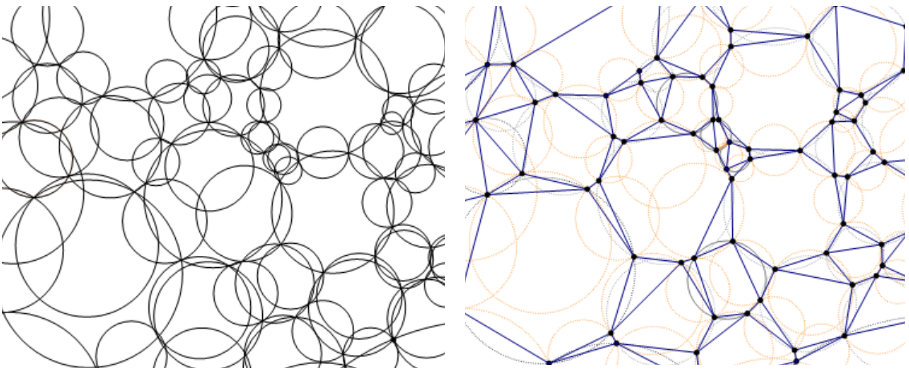} 

\caption{A local picture of a Delaunay circle pattern and its graph.}
\label{fig-2}
}
\end{figure}
\begin{figure}[H]
{
\centering
\includegraphics[height=5cm]{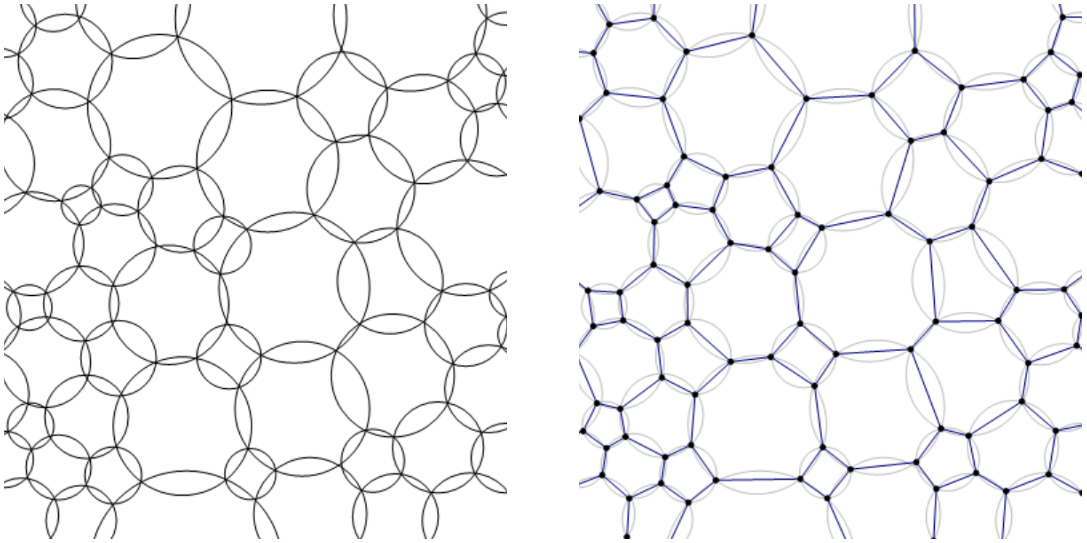} 

\caption{A local picture of an ideal circle pattern and its graph.}
\label{fig-5}
}
\end{figure}
The cone angle $\alpha_v$ of $(S, \mathcal{D})$ at a vertex $v\in V_{\mathcal{D}}$ is determined by the weights $\theta_e$ as 
\begin{equation}\label{sing at a vertex}
\alpha_v=\sum _{e: \text{ edge from }v} (\pi - \theta_e),
\end{equation}
and the cone angle $\alpha_f$ at the center of a face $D_f$ satisfies the following condition
\begin{equation}\label{condi on sings of faces}
\sum_{f\in F'}\alpha_f < \sum_{e\in E(F^{\prime})} 2\theta_e \;\text{ for any } \;F^{\prime}\subset F\text{ and }F^{\prime}\neq\emptyset,
\end{equation} 
where $E(F^{\prime})=\{e\in E| \;e \text{ is an edge of a face }f\in F^{\prime}\}.$
    
Using Colin de Verdi`ere’s variation method (\cite{MR1106755}), Bobenko and Springborn (\cite{MR2022715}) showed that there exists a unique hyperbolic $(S, \mathcal{D})$ realizing a prescribed weighted graph and with prescribed cone angles, as long as the prescription satisfies the above condition (\ref{condi on sings of faces}).
\begin{theorem}[Bobenko-Springborn]\label{Rigidity pattern}
     Let $S$ be a closed topological oriented surface and $G$ be a graph on $S$ with the sets of edges and faces denoted by $E$ and $F$, such that each $e \in E$ is weighted by some $\theta_e \in\left(0, \frac{\pi}{2}\right]$ and every $f \in F$ is simply connected. Then $\langle\alpha_f\rangle_{f \in F} \in \mathbb{R}_{+}^{F}$ satisfies (\ref{condi on sings of faces}) if and only if there exists a hyperbolic metric $\sigma$ on $S$ with conical singularities, along with a circle pattern $\mathcal{D}$ on ($S, \sigma$), such that $G_{\mathcal{D}}$ coincides with $G$ (up to isotopy) as weighted graphs and $\sigma$ has cone angle $\alpha_f$ at the center of the disk $D_f \in \mathcal{D}$ for each $f \in F$. Moreover, $(S, \mathcal{D})$ is unique up to isotopy in the sense that if there are two such circle patterns
      $(S, \mathcal{D}_1)$ and  $(S, \mathcal{D}_2)$ then there exists an isometry $\phi: (S, \mathcal{D}_1) \rightarrow (S, \mathcal{D}_2)$ homotopic to the identity map. 
\end{theorem}

It is pointed out in \cite{MR4683863} that a similar result of Theorem \ref{Rigidity pattern}
holds for a closed oriented surface equipped with a Euclidean metric (with possible singularities at vertices or centers of faces), which may be viewed as a discrete version of Theorem B in \cite{Troyanov1991}. But the uniqueness in this case is up to a global rescaling in the sense that if there are two such circle patterns
      $(S, \mathcal{D}_1)$ and  $(S, \mathcal{D}_2)$ then there exists a homeomorphism $\phi: (S, \mathcal{D}_1) \rightarrow (S, \mathcal{D}_2)$ homotopic to the identity map and 
      $$d_2(\phi(x), \phi(y))=cd_1(x, y)$$
      for a constant $c>0$ and any two points $x$ and $y$ on $S$. It is a common guidance in the field that any method used to show Theorem \ref{Rigidity pattern} (in the hyperbolic case) can be modified to prove the analogue in the Euclidean case.

However, a direct analogue of Theorem \ref{Rigidity pattern} fails on a closed oriented surface with a conical spherical metric, since M\"obius transformations on the sphere preserve angles but distort the metric. To obtain a similar result of rigidity up to isotopy, Nie (\cite{MR4683863}) uses an identical condition of (\ref{condi on sings of faces}) except to replace $\alpha_f$ by the total geodesic curvature of the circle corresponding to each face $f$, and applies Colin de Verdi`ere’s variation method (\cite{MR1106755}) to achieve the existence and uniqueness (up to isotopy). 

In this paper, we first apply the classical Perron method to give an alternative proof of Nie's result. Then we show the convergence of Thurston's algorithm (\cite{Thurston1979TheGA}, pages 13.44ff.), which adjusts the geodesic curvatures of circles one by one based on the prescribed values for total geodesic curvatures of the circles, to the desired circle pattern in the setting of Nie's result. 

Let us also point out that the classical Perron method can be used to develop alternative proofs of Theorem \ref{Rigidity pattern} and its analogue in the Euclidean case. 

In the remainder of this introduction section, we introduce some necessary background and give precise statements of the results which we prove in this paper.

\subsection{Closed surface with conic spherical metric}
Assume that $\mathbb{S}^2$ stands for the unit sphere in $\mathbb{R}^3$ centered at the oirgin. 
Let $\mathbb{S}^2_\alpha$ be obtained by gluing the two sides of a spherical lune on $\mathbb{S}^2$ of an angle $\alpha$, and let $D_\alpha(r)$ be the open metric disk of radius $r$ on $\mathbb{S}_\alpha^2$ centered at the north pole $A\in \mathbb{S}_\alpha^2$, where $r\in (0, \pi)$, which we call {\em a conical spherical disk} of angle $\alpha $ and radius $r$. Note that $r$ may be taken in the range $(0, \pi)$, {\em but in this paper, we restrict $r\in (0, \frac{\pi}{2})$.} On Figure \ref{fig-8}, $D_{\alpha_1}(r_1)$ shows an example of $D_\alpha(r)$.
\begin{figure}[H]
{
\centering
\includegraphics[height=5cm]{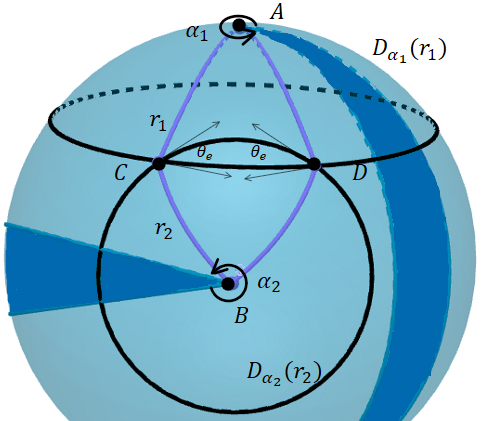} 

\caption{A conical spherical disk.}
\label{fig-8}
}
\end{figure}

More generally, for any $\alpha\in(0, +\infty)$ and any $r\in (0, \pi)$, a \textit{conical spherical disk} $D_{\alpha}(r)$ of angle $\alpha $ and radius $r$ can be defined in spherical background geometry, which is often called a {\em cone} of angle $\alpha $ and radius $r$, the center is called the {\em conical point} and $\alpha$ is called the {\em cone angle}. In this paper, keep in mind that the radius \( r \) of the conical disk lies in the interval \( (0, \frac{\pi}{2}) \). We may briefly sketch a conical spherical disk as the one in Figure \ref{fig-3}. 

\begin{figure}[H]
{
\centering
\includegraphics[height=3cm]{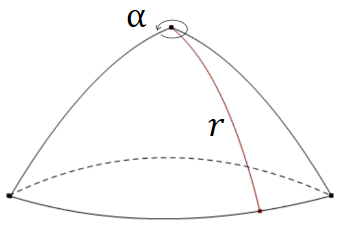} 

\caption{A conical spherical disk $D_{\alpha}(r)$.}
\label{fig-3}
}
\end{figure}

The {\em discrete Gaussian curvature} of $D_\alpha(r)$ at the conic point  is defined as $$K=2\pi-\alpha.$$ 
Denote by $k$ and $\ell$ the geodesic curvature and the circumference of the boundary curve $\partial  D_{\alpha}(r)$ of $D_{\alpha}(r)$, and denote by $Area(D_{\alpha}(r))$ the area of $D_{\alpha}(r)$. Then 
\begin{equation}\label{equ-2}
k=\cot r,\quad \ell=\alpha\sin r,\quad Area(D_{\alpha}(r))=2\alpha\sin^2\frac{r}{2}.
\end{equation}
From the Gauss-Bonnet formula, 
\begin{equation}\label{equ-1}
Area(D_{\alpha}(r))+k\cdot \ell+K=2\pi\chi(D_{\alpha}(r)).
\end{equation}
The total {\em geodesic curvature} of $\partial  D_{\alpha}(r)$ is defined as $$T=k\cdot \ell=\alpha\cos r. $$

Let $S$ be a closed oriented surface. We say that $S$ is equipped with a {\em conical spherical metric} if there is a collection $\{\phi_i: D_{\alpha_i}(r_i)\rightarrow S\}_{i\in I}$ of coordinate charts such that 

\begin{itemize}
    \item $\{\phi_i(D_{\alpha_i}(r_i))\}_{i\in I}$ is a covering of $S$ except finitely many points, and 

    \item  for any $i, j\in I$ with $i\neq j$, if $\phi_i(D_{\alpha_i}(r_i))\cap \phi_j(D_{\alpha_j}(r_j))\neq \emptyset$, then the transition map 
$$ \phi_j\circ (\phi_i)^{-1}: \phi_i(D_{\alpha_i}(r_i))\rightarrow \phi_j(D_{\alpha_j}(r_j))$$
is an isometry in the spherical metric. 

\end{itemize}

\subsection{Existence and uniqueness of ideal circle pattern on a closed surface equipped with conical spherical metric}
Let $S$ be a closed oriented surface equipped with a conical spherical metric $\sigma$ and $\{x_1,\cdots,x_n\}$ be a set of finitely many points on $S$, and let $\mathcal{D}=\{D_1,D_2,\cdots,D_m\}$ be a collection of finitely many conical spherical disks. A local isometry $\iota:\sqcup_{i=1}^m D_i\rightarrow S$ from the disjoint union of the disks in $\mathcal{D}$ to $S$ is called a \textit{spherical ideal circle pattern} on $S$ if the following conditions are satisfied:

(1) $\iota (\sqcup_{i=1}^m D_i)=S\backslash\{x_1,\cdots,x_n\} $;

(2) $\iota $ maps at most two to one;

(3) each $\iota(D_i)$ is not contained in any other $\iota(D_j)$.


Let $G_{\mathcal{D}}=\{V_{\mathcal{D}},E_{\mathcal{D}},F_{\mathcal{D}}\}$ be the graph determined by an ideal spherical circle pattern $\mathcal{D}$ on a closed oriented surface $S$ (a local picture of $G_{\mathcal{D}}$ is illustrated on Figure \ref{fig-5}). The graph embedded on the surface $S$ gives a polyhedral decomposition of $S$. There is a bigon $\Omega_e$ corresponding to each edge $e\in E$ and we define the angle of the bigon as the weight of the edge. Moreover, this ideal spherical circle pattern $\mathcal{D}$ determines the conical spherical metric $\sigma $ on $S$, with possible singularities at some vertices and some centers of the disks, so that the map $\iota$ is a local isometry on each disk $D_i$. More precisely, the following theorem is obtained in \cite{MR4683863}, which we give a different proof in this paper.

\begin{theorem}[\cite{MR4683863}]\label{Rigidity theorem} Let $S$ be a closed oriented surface and $G$ be a graph on $S$ with the sets of edges and faces denoted by $E$ and $F$. Assume that the weights on the edges of $E$ are given by a map $\Theta: E\rightarrow (0, \pi/2]$. Let $\hat{{\bf{T}}}: F\rightarrow (0, +\infty)$ be a map defined on the face set $F$. The following two statements are equivalent:

(i) The map $\hat{\bf T}\in\mathcal{T}$, where\begin{equation}\label{equ-3}\mathcal{T}=\left\{\begin{array}{l|l}{\bf{T}}: F\rightarrow (0, +\infty)\end{array}\text{ }\sum_{f\in F^{\prime}}{\bf{T}}(f)<\sum_{e\in E(F^{\prime})} 2\Theta(e), \forall \;F^{\prime}\subset F\text{ and }F^{\prime}\neq\emptyset\right\}\end{equation} and $E(F^{\prime})=\{e\in E| \;e \text{ is an edge of an face }f\in F^{\prime}\}.$

(ii) There exists a unique spherical metric $\sigma$ with conical singularities on $S$ along with an ideal spherical circle pattern $\hat{\mathcal{D}}$ such that the corresponding graph $G_{\mathcal{D}}$ coincides with $G$ up to isotopy and the total geodesic curvature of the boundary circle of the conical disk $D_f$ is equal to $\hat{{\bf{T}}}(f)$ for every $f\in F$. 
\end{theorem}

\subsection{Convergence of Thurston algorithm}\label{sec-1.5}
Thurston presented an elegant algorithm for finding a circle packing with a prescribed data in \cite{Thurston1979TheGA}. In this paper, we show that this algorithm can be used to approximate the circle pattern in Theorem \ref{Rigidity theorem} as well.

Using the same notation as introduced in Theorem \ref{Rigidity theorem}, let $\hat{{\bf{T}}}\in \mathcal{T}$ and let
$$\langle \hat{T}_1,\hat{T}_2,\cdots,\hat{T}_{|F|}\rangle=\langle \hat{\textbf {T}}(f_1),\hat{{\bf{T}}}(f_2),\cdots,\hat{{\bf{T}}}({f_{|F|})}\rangle .$$
We call the circle pattern in (ii) of Theorem \ref{Rigidity theorem} the target circle pattern and denote it by $\hat{\mathcal{D}}$. Let $\hat{\textbf{k}}: F\rightarrow (0,+\infty)$ be the map representing the vector of the geodesic curvatures of the boundary circles of the disks of $\hat{\mathcal{D}}$ and let 
$$
\langle \hat{k}_1,\hat{k}_2,\cdots ,\hat{k}_{|F|}\rangle=\langle\hat{\textbf{k}}(f_1),\hat{\textbf{k}}(f_2),\cdots, \hat{\textbf{k}}(f_{|F|})\rangle .
$$
Now we describe how to apply Thurston's algorithm to find the target circle pattern $\hat{\mathcal{D}}$. 

{\bf Step 1.} Let $\textbf{k}^0: F\rightarrow (0,+\infty)$ be an arbitrarily chosen initial vector of the geodesic curvatures of the boundary circles of the disks corresponding to the faces of $F$. From Proposition 6 in \cite{MR4683863} (which we recapitulate its proof at the beginning of Section \ref{sec-1.5}), there is a unique ideal spherical circle pattern $\mathcal{D}^{0}$ (up to isometry) underlying $(S, G,\Theta)$ such that the geodesic curvature of the circles in $\mathcal{D}^{0}$ satisfies
$$
\langle k_1^0,k_2^0,\cdots ,k_{|F|}^0\rangle=\langle\textbf{k}^0(f_1),\textbf{k}^0(f_2),\cdots, \textbf{k}^0(f_{|F|})\rangle 
$$

Here, $k^0_i$ is defined as the geodesic curvature of the circle in $\mathcal{D}^0$ corresponding to the face $f_i$ for each $i=1,2,\cdots,|F|$. Hence, the total geodesic curvatures of the generalized circles is given by

 $$
\langle T^0_1,T^0_2,\cdots,T^0_{|F|}\rangle=\langle{\bf{T}}^0(f_1),{\bf{T}}^0(f_2),\cdots,{\bf{T}}^0(f_{|F|})\rangle .
$$

Here, $\mathbf{T^0}:F\rightarrow(0,+\infty)$ is determined by $\mathbf{k^0}$. By Lemma\ref{lem-2.5} and Lemma \ref{lem-2.6}, we have $\mathbf{T^0}\in \mathcal{T}$.

{\bf Step 2.} For each $1\le i\le |F|$, We take the following action:
without changing the angles of any two intersecting circles and without changing the geodesic curvature of the boundary circle of any other disk, we adjust the geodesic curvature $k^0_i$ of $\partial D^0_i$ to $k^1_i$ by deforming $D_i^0$ to $D_i^1$ so that the total geodesic curvature $T_i^1$ of $\partial D_i^1$ is equal to $\hat{T}_i$ (keep in mind that this action may change the length and hence the total geodesic curvature of the boundary circle of any disk $D^0_j$ intersecting $D^0_i$). Define a new vector of geodesic curvatures $\textbf{k}^1:F\rightarrow (0,+\infty)$ by 
$$\langle\textbf{k}^1(f_1),\textbf{k}^1(f_2),\cdots, \textbf{k}^1(f_{|F|})\rangle=\langle k_1^1,k_2^1,\cdots ,k_{|F|}^1\rangle.$$ 

{\bf Step 3.} Replace $\textbf{k}^0$ by $\textbf{k}^1$ and repeat Steps 1 and 2. 

Applying Steps 1, 2 and 3 inductively, we obtain a sequence $\{\mathcal{D}^m\}_{m=0}^\infty$ of ideal spherical circle patterns underlying $(S, G,\Theta)$ such that 
$$
\langle k_1^m,k_2^m,\cdots ,k_{|F|}^m\rangle=\langle\textbf{k}^m(f_1),\textbf{k}^m(f_2),\cdots, \textbf{k}^m(f_{|F|})\rangle 
$$
and
 $$
\langle T^m_1,T^m_2,\cdots,T^m_{|F|}\rangle=\langle{\bf{T}}^m(f_1),{\bf{T}}^m(f_2),\cdots,{\bf{T}}^m(f_{|F|})\rangle \in \mathcal{F} .
$$
We simply call $\mathcal{D}^m$ the $m^{th}$ approximation of the target pattern $\hat{\mathcal{D}}$. We prove the following convergence theorem. 

\begin{theorem}\label{Thurston's algorithm}
Let $S$ be a closed oriented surface with an embedded weighted graph $(G,\Theta)$, where
\begin{itemize}
    \item $G=\{V,E,F\}$ denotes the graph's vertex set $V$, edge set $E$ and face set $F$, and 
    \item $\Theta:E\rightarrow (0,\frac{\pi}{2}]$ assigns weights to the edges.
\end{itemize}
Assume that $\hat{{\bf{T}}}: F\rightarrow (0, \infty)$ lies in $\mathcal{T}$. Then for any initial ideal circle pattern $\mathcal{D}^0$ underlying $(S,G,\Theta)$ with the corresponding total geodesic curvature vector ${\bf{T}}^0\in \mathcal{T}$, the circle pattern sequence $\{\mathcal{D}^m\}_{m=1}^\infty$, constructed from the Thurston algorithm, converges to $\hat{\mathcal{D}}$ with $\lim_{m\rightarrow\infty}{\bf{k}}^m =\hat{\bf{k}} $ and $\lim_{m\rightarrow\infty }{\bf{T}}^m =\hat{\bf{T}}$. 

Furthermore, if $\bf{k}^0$ satisfies ${\bf{T}}^0\le \hat{\bf{T}}$, then ${\bf{k}}^m(f)$ increases to $\hat{\bf{k}}(f)$ as $m\rightarrow \infty$ for any $f\in F$ and ${\bf{T}}^m\le \hat{\bf{T}}$ for any $m$; if $\bf{k}^0$ satisfies ${\bf{T}}^0\ge \hat{\bf{T}}$, then ${\bf{k}}^m(f)$ decreases to $\hat{\bf{k}}(f)$ as $m\rightarrow \infty$ for any $f\in F$ and ${\bf{T}}^m\ge \hat{\bf{T}}$ for any $m$.
\end{theorem}

\section{Preliminary}
We first state the well-known spherical contangent 4-part formula. 
\begin{figure}[H]
{
\centering
\includegraphics[height=3cm]{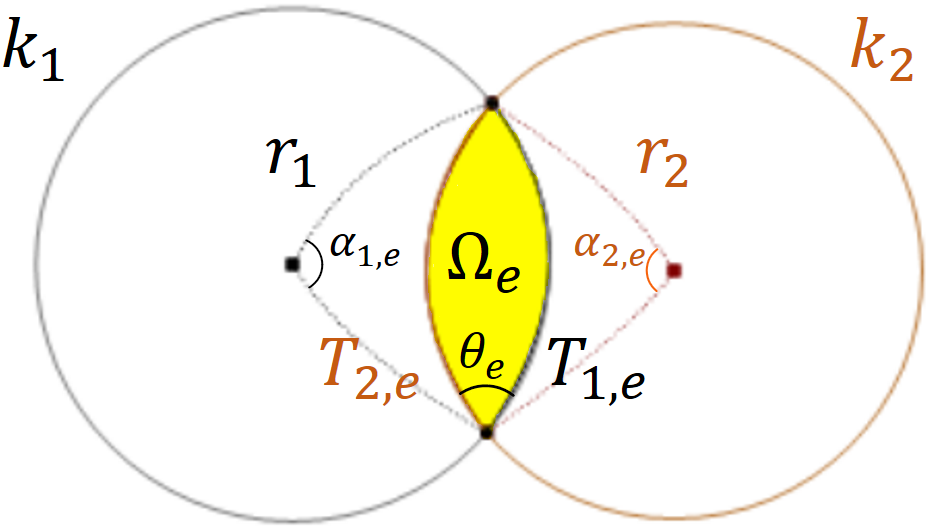} 

\caption{Two intersecting conical spherical disks.}
\label{fig-6}
}
\end{figure}
\begin{lemma}[Spherical cotangent 4-part formula]\label{cotagent 4-part formula}
Let $D_{2\pi}(r_i)$, $i=1,2$, be two conical spherical disks intersecting at an angle $\theta_e\in(0,\frac{\pi}{2}]$, and let $\alpha_{i,e}$ be the measure (in radians) of the central angle of the edge $e$ on $D_{2\pi}(r_i)$ (see Figure \ref{fig-6}). Then
$$
\cot \frac{\alpha_{i,e}}{2}=\frac{1}{\sin\theta_e}(\cot r_{j}\cdot \sin r_i+\cos r_i\cos\theta_e),
$$
where $\{i, j\}=\{1, 2\}$.
\end{lemma}

Now we state two facts given in \cite{MR4683863}.

\begin{lemma}[Proposition 2 of \cite{MR4683863}]\label{lem-2.1}
Let $S$ be a closed oriented surface equipped with a conical spherical metric $\sigma$ and $\mathcal{D}=\{D_1,D_2,\cdots,D_m\}$ be an ideal spherical circle pattern on $S$. Let $\iota:\sqcup_{i=1}^m D_i\rightarrow S$ be a local isometry from the disjoint union of the disks in $\mathcal{D}$ to $S$. Set $B=\{p\in S: \iota^{-1}(p) \text{ has two points }\}$. Then $\iota^{-1}(B)$ is a disjoint union of bigons, which fills up the boundary of each $D_i$, and the equivalence relation ``$\sim$", defined by $x\sim y \leftrightarrow \iota(x)=\iota(y)$, induces isometric pairings among the bigons. 
\end{lemma}

\begin{lemma}[Monotonicity \cite{MR4683863}]\label{lem-2.2}
Let $D_i$, $i=1,2$, be two conical spherical disks intersecting at an angle $\theta_e\in(0,\frac{\pi}{2}]$. Denote by $C_i$, $i=1,2$, their boundary circles with radii $r_i\in (0,\frac{\pi}{2})$, $i=1,2$. The corresponding geodesic curvatures are given by $k_i=\cot r_i\in (0,+\infty)$. Let $u_i=\ln k_i \in (-\infty,+\infty)$. The intersection forms a bigon $\Omega_e$ (as shown in Figure \ref{fig-6}).  For $i=1,2$, we define $T_{i,e}$ as the total geodesic curvature of this boundary arc. Then the following properties hold.

\begin{itemize}
    \item [(a)] $\frac{\partial  T_{i,e}}{\partial  u_i}>0, i=1,2$.\\
    \item [(b)] $\frac{\partial  T_{i,e}}{\partial  u_j}=\frac{\partial  T_{j,e}}{\partial  u_i}<0, \{i,j\}=\{1,2\}$.\\
    \item [(c)] $\frac{\partial  T_{i,e}}{\partial  u_i}+\frac{\partial  T_{i,e}}{\partial  u_j}>0,\{i,j\}=\{1,2\}$.\\
    \item [(d)] $\frac{\partial  Area(\Omega_e)}{\partial  u_i}<0,i=1,2$.
\end{itemize}

\end{lemma}

The previous Lemma \ref{lem-2.2} is a key for the proof of Theorem \ref{Rigidity theorem} given in \cite{MR4683863}, which applies Colin de Verdi`ere’s variation method. It is also a key for us to develop
a proof for this theorem by using the classical Perron method. So, we present a slightly simpler proof of this lemma here. 
\begin{proof} Using the spherical $4$-part cotangent formula (Lemma \ref{cotagent 4-part formula}), we obtain
$$
\cot \frac{\alpha_{i,e}}{2}=\frac{1}{\sin\theta_e}(\cot r_{j}\cdot \sin r_i+\cos r_i\cos\theta_e)=\frac{1}{\sin\theta_e}\cdot\frac{k_j+k_i\cos\theta_e}{\sqrt{k_i^2+1}}.
$$
Thus, $$\alpha_{i,e}=2 \text{arccot} \left( \frac{1}{\sin\theta_e}\cdot\frac{k_j+k_i\cos\theta_e}{\sqrt{k_i^2+1}}\right)$$
and $$
T_{i,e}=k_i\cdot l_{i,e}=\alpha_{i,e}\cos r_i= \frac{2k_i}{\sqrt{k_i^2+1}}\text{arccot} \left( \frac{1}{\sin\theta_e}\cdot\frac{k_j+k_i\cos\theta_e}{\sqrt{k_i^2+1}}\right).
$$ 
Clearly,
$$\frac{\partial T_{i,e}}{\partial u_{j}}=-\frac{2k_ik_j\sin\theta_e}{k_i^2+k_j^2+2k_ik_j\cos\theta_e+\sin^2\theta_e}<0.$$
By a symmetric argument, we can see that $\frac{\partial T_{j,e}}{\partial u_{i}}<0.$ We have proved (b).

Note that if $k_j$ is fixed, then the area of the bigon $\Omega_e$ decreases as $k_i$ increases. This implies (d); that is, $\frac{\partial  Area(\Omega_e)}{\partial  u_i}<0.$

By the Gauss-Bonnet formula, 
\begin{equation}\label{Gauss-Bonnet formula on the quadrilateral}
 Area(\Omega_e)=2\theta_e-T_{i,e}-T_{j,e}.   
\end{equation}
Then 
$$\frac{\partial Area(\Omega_e)}{\partial u_i}=-\left(\frac{\partial T_{i,e}}{\partial u_i}+\frac{\partial T_{j,e}}{\partial u_i}\right)<0.$$
Thus,
$$
\frac{\partial T_{i,e}}{\partial u_i}+\frac{\partial T_{i,e}}{\partial u_j}=\frac{T_{i,e}}{\partial u_i}+\frac{T_{j,e}}{\partial u_i}>0,
$$ 
which proves (c). 

It follows that 
$$\frac{\partial T_{i,e}}{\partial u_i}=(\frac{\partial T_{i,e}}{\partial u_i}+\frac{\partial T_{i,e}}{\partial u_j})+(-\frac{\partial T_{i,e}}{\partial u_j})>0,$$
which proves (a).
\end{proof}

The first two properties in Lemma \ref{lem-2.2} establish the strict monotonicity of the total geodesic curvature $T_{i,e}$ in the sense that it increases monotonically with $u_i$ (or $k_i$) but decreases with $u_j$ (or $k_j$), where ${i,j} = {1,2}$; the third property shows that $u_i$ dominates $u_j$ in determining $T_{i,e}$, which is a key for us to prove the convergence of the modified Thurston algorithm in Section \ref{sec-4}; the fourth property reveals an inverse proportionality between the area of $\Omega_e$ and the curvature parameters $k_1$ and $k_2$.

\begin{lemma}\label{lem-2.3}
Let $f$ be a polygon with n vertices $\langle v_1,v_2,\cdots,v_n\rangle$ and n weighted edges $\langle(e_1,\theta_1),(e_2,\theta_2),\cdots,(e_n,\theta_n)\rangle$. Suppose that the geodesic curvature of the conical circle $C_f$ corresponding to $f$ is $k_f$ and the total geodesic curvature is $T_f$. Let $k_i$ be the geodesic curvature of the neighboring conical circle against the edge $e_i$,$i=1,2,\cdots,n$ (as shown in Figure \ref{fig-7}). Then $T_f$ is uniquely determined by $k_f$ and $k_1,k_2,\cdots,k_n$. Moreover, $T_f(k_f,k_1.\cdots,k_n)$ increases strictly with $k_f$ and decreases strictly with $k_i$ for each $i=1,\cdots, n$.

\end{lemma}

\begin{figure}[H]
{
\centering
\includegraphics[height=6cm]{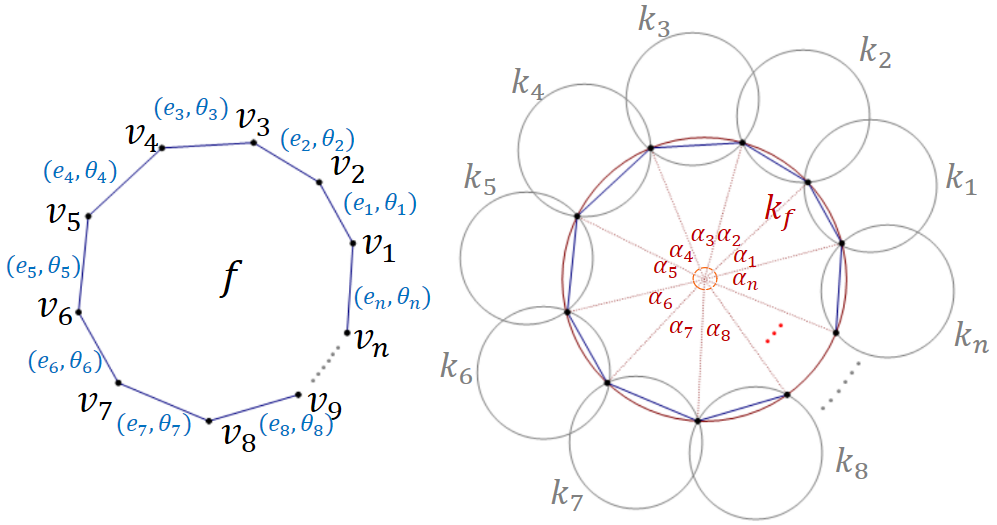} 

\caption{Spherical circle pattern on a polygon.}
\label{fig-7}
}
\end{figure}

\begin{proof}
Using the spherical cotangent 4-part formula (see Lemma \ref{cotagent 4-part formula}), we obtain
\begin{equation}\label{equ-5}
    \cot\frac{\alpha_i}{2}=\frac{1}{\sin \theta_{i}}(\cot r_i\sin r_f+\cos r_f\cos \theta_{i}),\quad i=1,2,\cdots,n.
\end{equation}
where $r_i$ is the radius of $C_i$, $k_i=\cot r_i$, and $r_f$ is the radius of $C_f$. Then 
$$
\begin{aligned}
\alpha_i=2\text{arccot}\frac{1}{\sin\theta_{i}}\cdot\frac{k_i+k_f\cos\theta_{i}}{\sqrt{k_f^2+1}}.
\end{aligned}
$$
Therefore,
$$
\begin{aligned}
T_f&=\sum_{i=1}^n T_{f,i}=k_f\cdot\sum_{i=1}^n\alpha_{i}\sin r_f\\
&=\frac{k_f}{\sqrt{1+k_f^2}}\sum_{i=1}^n2\text{arccot}(\frac{1}{\sin\theta_{i}}\cdot\frac{k_i+k_f\cos\theta_{i}}{\sqrt{k_f^2+1}}).
\end{aligned}
$$
Clearly, 
the function $T_f(k_f,k_1.\cdots,k_n)$ is continuous and differentiable. From  (a) and (b) of Lemma \ref{lem-2.2}, it follows that $\frac{\partial T_f}{\partial k_f}>0$ and $\frac{\partial T_f}{\partial k_i}<0$ for each $1\leq i\leq n$.

\end{proof}

In order to show the nonemptiness of the set $\mathcal{T}$ defined by (\ref{equ-3}) in Theorem \ref{Rigidity theorem}, we consider the limit behavior of the total geodesic curvature as the geodesic curvature goes to $0$ or $\infty $.

Let $\langle c_1,c_2\rangle$ be a nonnegative vector, where $c_i\in[0,+\infty]$, $i=1,2$.

\begin{lemma}\label{lem-2.5} Let $i=1,2$. 
If $c_i=0$, then $\lim _{\langle k_1,k_2 \rangle\rightarrow \langle c_1,c_2 \rangle} T_{i,e}=0$.
\end{lemma}

\begin{proof} Let $i=1,2$.
If $c_i=0,c_j\neq 0$ for some $i, j=1,2$ and $j\neq i$, then 
$$
\begin{aligned}
\lim _{\langle k_1,k_2 \rangle\rightarrow \langle c_1,c_2 \rangle} T_{i,e}&=\lim_{\langle k_1,k_2 \rangle\rightarrow \langle c_1,c_2 \rangle}\frac{k_i}{\sqrt{1+k_i^2}}\cdot 2 \text{arccot}(\frac{1}{\sin \theta_e}\cdot\frac{k_j+k_i\cos\theta_e}{\sqrt{1+k_i^2}})\\
&=\lim_{\langle k_1,k_2 \rangle\rightarrow \langle c_1,c_2 \rangle}k_i\cdot 2 \text{arccot}(\frac{k_j}{\sin\theta_e})\\
&=0.\end{aligned}$$

If $c_1=c_2=0$, then 
$$
\begin{aligned}
\lim _{\langle k_1,k_2 \rangle\rightarrow (0,0)} T_{1,e}=\lim _{\langle k_1,k_2 \rangle\rightarrow (0,0)} T_{2,e}=0.
\end{aligned}$$

\end{proof}

\begin{lemma}\label{lem-2.6}
If $c_i=\infty$ for some $i=1,2$, then $$\lim _{\langle k_1,k_2 \rangle\rightarrow \langle c_1,c_2 \rangle} T_{i,e}+T_{j,e}=2\theta_e.$$

\end{lemma}

\begin{proof}
If $c_i=\infty,c_j<\infty$ for some $i, j=1,2$ and $i\neq j$, then 
$$
\begin{aligned}
\lim _{\langle k_1,k_2 \rangle\rightarrow \langle c_1,c_2 \rangle} T_{i,e}&=\lim_{\langle k_1,k_2 \rangle\rightarrow \langle c_1,c_2 \rangle}\frac{k_i}{\sqrt{1+k_i^2}}\cdot 2 \text{arccot}(\frac{1}{\sin \theta_e}\cdot\frac{k_j+k_i\cos\theta_e}{\sqrt{1+k_i^2}})\\
&=2\text{arccot}(\frac{\cos\theta_e}{\sin\theta_e})\\
&=2\theta_e
\end{aligned}
 $$
and $$\lim _{\langle k_1,k_2 \rangle\rightarrow \langle c_1,c_2 \rangle} T_{j,e} =0.$$
Therefore,

$$
\lim_{\langle k_1,k_2 \rangle\rightarrow\langle c_1,c_2 \rangle}T_{1,e}+T_{2,e}=2\theta_e.
$$

If $c_1=c_2=\infty$, then by the Gauss-Bonnet formula, 

$$
\lim_{\langle k_1,k_2 \rangle\rightarrow(\infty,\infty)}T_{1,e}+T_{2,e}=2\theta_e-\lim_{(k_1.k_2)\rightarrow(\infty,\infty)}Area(\Omega_e)=2\theta_e.
$$

\end{proof}

\section{Proof of Theorem \ref{Rigidity theorem}}
\begin{proof} We first prove that (ii) implies (i).

Let us recapitulate the proof of Proposition 6 in \cite{MR4683863} to see that for any map $\textbf{k}:F\rightarrow(0,+\infty)$, there is a unique ideal spherical circle pattern $\mathcal{D}_{\textbf{k}}$ (up to isometry) underlying $(S, G,\Theta)$ such that $k_f=\textbf{k}(f)$ is the geodesic curvature of the boundary circle of the disk corresponding to a face $f\in F$. 

For two intersecting conical spherical disks $D_1$ and $D_2$, there is a quadrilateral $Q$ bounded by two radii of length $r_1$ and two radii of length $r_2$, with vertices at two intersecting points and the centers of the two disks, which is uniquely determined by $r_1$, $r_2$ and weight $\theta_e$ (as shown in Figure \ref{fig-13}). 
By gluing the intersecting portion, we obtain a spherical metric on the quadrilateral $Q$. Given a vector $\langle r_f\rangle_{f\in F}\in (0,\frac{\pi}{2})^{|F|}$, we first obtain a quadrilateral $Q_{i, j}$ from any two intersecting disks $D_i$ and $D_j$. Then we glue adjacent quadrilaterals along their common edges according to the graph $G$. Through this process, we obtain a conical spherical metric on $S$ underlying $(G,\Theta)$. It follows that there is an ideal circle pattern (unique up to isometry) such that $r_f$ is the radius of the boundary circle of the disk $D_f$ for each $f\in F$. Since $k_f=\cot r_f$, it induces a bijection between $\langle r_f\rangle_{f\in F}$ and $\langle k_f\rangle_{f\in F}$. Therefore, given $\langle k_f\rangle_{f\in F}\in (0,+\infty)^{|F|}$, there is an ideal circle pattern (unique up to isometry) underlying $(S, G,\Theta)$ such that $k_f$ is the geodesic curvature of the boundary circle of $D_f$ for each $f\in F$.

\begin{figure}[H]
{
\centering
\includegraphics[height=4cm]{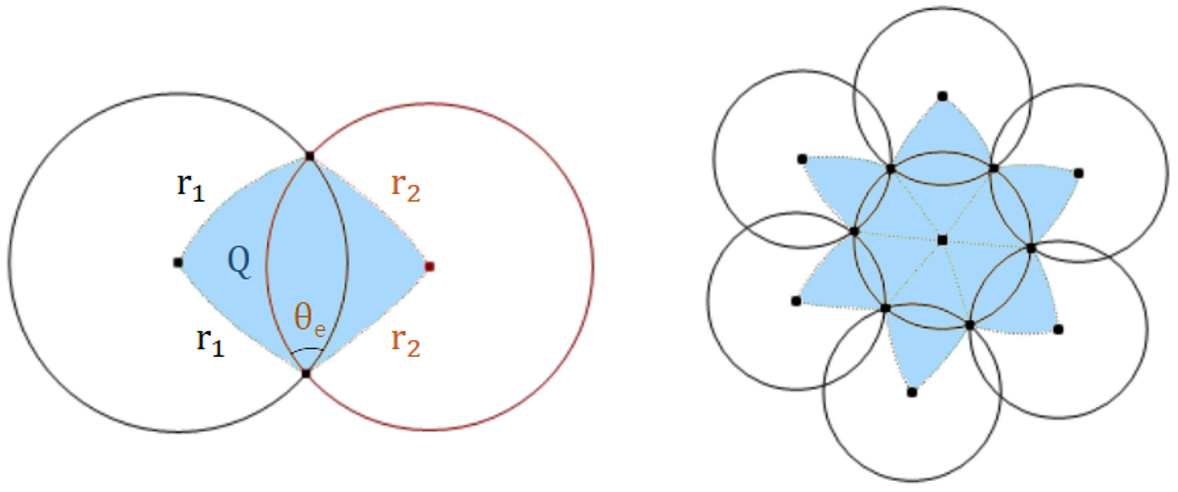} 
\caption{Gluing the quadrilaterals.}
\label{fig-13}

}
\end{figure}
 
Let ${\bf{T}}_{\textbf{k}}:F\rightarrow(0,+\infty)$ represent the vector of the total geodesic curvatures of the circles of the disks of $\mathcal{D}_{\textbf{k}}$. Using the restrain $r\in (0, \frac{\pi}{2})$ and the Gauss-Bonnet formula (\ref{Gauss-Bonnet formula on the quadrilateral}), we know that
$$0<(T_{\textbf{k}})_{f,e}+(T_{\textbf{k}})_{g,e}<2\theta_e$$
whenever two faces $f$ and $g$ share an edge $e$, where $(T_{\textbf{k}})_{f,e}$ and $(T_{\textbf{k}})_{g,e}$ are defined in Lemma \ref{lem-2.2}. It follows that
$$0<(T_{\textbf{k}})_{f,e}<2\theta_e$$ for any $f\in F$ and $e\in E(f)$. Thus, for any nonempty subset $F^{\prime}\subset F$, the total geodesic curvatures satisfy that
$$ 
0<\sum_{f\in F^{\prime}} {\bf{T}}_{\textbf{k}}(f)<\sum_{e\in E(F^{\prime})} 2\Theta(e).
$$
Therefore, ${\bf{T}}_{\textbf{k}}\in \mathcal{T}$. It follows that $\hat{\bf{T}}\in \mathcal{T}$.

Now we apply the Perron method instead of the discrete variation principle to prove that (i) implies (ii). We first outline the proof.

Given a map $\hat{{\bf{T}}}:F\rightarrow (0, \infty )$ in $\mathcal{T}$, we define
$\mathcal{K}_1$ as
\begin{equation}\label{equ-6}
\mathcal{K}_1=\left\{\begin{array}{l|l}\textbf{k}:F\rightarrow(0,+\infty)\end{array}\begin{array}{l}
\mathcal{D}_{\textbf{k}}\text{ is a spherical ideal circle}\\ \text{pattern underlining }(S,G,\Theta)\\\text{ with } {\bf{T}}_{\textbf{k}}(f)\leq\hat{{\bf{T}}}(f),\forall f\in F\end{array}\right \}.
\end{equation}
We call $\mathcal{K}_1$ the subpattern set of $\hat{{\bf{T}}}$. 

We first show that $\mathcal{K}_1$ is not empty. Then we prove the subpattern set $\mathcal{K}_1$ has so-called net properties. Define $\sup \mathcal{K}_1$ to be the map $\boldsymbol{\tau}_1: F\rightarrow (0, \infty)$ by setting 
\begin{equation}\label{equ-8}
\boldsymbol{\tau}_1(f)=\sup_{\textbf{k}^{\prime}\in \mathcal{K}_1}\textbf{k}^{\prime}(f),\forall f_i\in F.
\end{equation}
Thirdly, we prove that $\mathcal{D}_{\boldsymbol\tau_1}$ is the target ideal circle pattern $\mathcal{D}$ such that the total geodesic curvature of the conical circle corresponding to each face $f_i\in F$ is $\hat{T}_i$. Finally, applying the monotonicity for the areas of bigons, we achieved the uniqueness of $\mathcal{D}$ via proof by contradiction. 

Now we provide the details for each step.

\textbf{Step I.} We show that $\mathcal{K}_1\neq \emptyset$.

Let $\mathcal{D}$ be an ideal spherical circle pattern underlying $(S, G,\Theta)$, and let
$\textbf{k}:F\rightarrow(0,+\infty)$ and $\textbf{T}:F\rightarrow(0,+\infty)$ represent
the vector of the geodesic curvatures and the vector of the total geodesic curvatures of the boundary circles of the disks corresponding to the faces $f\in F$. 
Using $\frac{\partial  T_{i,e}}{\partial  k_i}+\frac{\partial  T_{i,e}}{\partial  k_j}>0$ ((c) of Lemma \ref{lem-2.2}) and Lemma \ref{lem-2.5}, one can choose the values of $\textbf{k}$ small enough so that all total geodesic curvatures of the corresponding spherical arcs $T_{i,e},\;e\in E(f_i),\;1\leq i\leq |F| $ are less than or equal to $\frac{m}{d}$, where $m=\min_{f_i\in F}{\hat{T}_i}$ and $d$ is the maximum number of edges ending at the vertex $v$ for all $v\in V$. Then $\textbf{k}\in\mathcal{K}_1$.

\textbf{Step II.} We show that $\mathcal{K}_1$ has the following net property: $\forall\;\textbf{k}^{\prime},\textbf{k}^{\prime\prime}\in\mathcal{K}_1$, define
\begin{equation}\label{equ-10}
\textbf{k}: F\rightarrow (0,+\infty):
    f\mapsto \max\{\textbf{k}^{\prime}(f),\textbf{k}^{\prime\prime}(f)\}.
\end{equation}
Let $\mathcal{D}$ be the ideal spherical circle pattern determined by $\textbf{k}$ and $\textbf{T}_{\textbf{k}}: F\rightarrow (0,+\infty)$ be the corresponding vector of the total geodesic curvatures. Then $\textbf{T}_{\textbf{k}}\in \mathcal{K}_1$.

Let $f\in F$ and assume that it has $n$ vertices on the weighted graph $(G,\Theta)$ (as shown in Figure \ref{fig-7}). Without loss of generality, we may assume that $\textbf{k}^{\prime}(f)\leq \textbf{k}^{\prime\prime}(f)$. Then $\textbf{k}(f)=\textbf{k}^{\prime\prime}(f)$. By the monotonicity property (b) in Lemma \ref{lem-2.2}, the total geodesic curvature of the circle corresponding to $f$ decreases as the geodesic curvatures of neighboring circles increase. Since $\textbf{k}(f_i)\geq\textbf{k}^{\prime\prime}(f_i)$ for the neighboring conical circles $C_i,i=1,\cdots,n$, it follows that ${\bf{T}}_{\textbf{k}}(f)\leq {\bf{T}}_{\textbf{k}^{\prime\prime}}(f)\leq\hat{{\bf{T}}}(f)$. Thus, $\textbf{T}_{\textbf{k}}\in \mathcal{K}_1$.

\textbf{Step III.} We show that $\mathcal{D}_{\boldsymbol{\tau}_1}$ is the desired spherical ideal circle pattern.

By Lemma \ref{lem-2.6}, we know that if $T_{i,e}+T_{j,e}<2\theta_e$, then $k_1,k_2<+\infty$. Therefore, given $\hat{{\bf{T}}}\in \mathcal{T}$,  $\textbf{k}(f)$ is uniformly bounded from above for all $\textbf{k}\in \mathcal{K}_1$ and all $f\in F$, where the bound depends on $\hat{{\bf{T}}}$. It follows that $\boldsymbol{\tau}_1=\sup\mathcal{K}_1: F\rightarrow (0, \infty)$ is well defined. Then there is an ideal spherical circle pattern $\mathcal{D}_{\textbf{$\boldsymbol{\tau}_1$}}$ underlying $(S,G,\Theta)$ with $\textbf{k}=\boldsymbol{\tau}_1$. Clearly, $\textbf{T}_{\boldsymbol{\tau}_1}\in \mathcal{K}_1$ implies that 
$\textbf{T}_{\boldsymbol{\tau}_1}(f)\le \hat{{\bf{T}}}(f)$ for each $f\in F$. We claim that $\textbf{T}_{\boldsymbol{\tau}_1}(f)=\hat{{\bf{T}}}(f)$ for each $f\in F$. Otherwise, suppose that there exists $f_0\in F$ such that ${\bf{T}}_{\boldsymbol{\tau}_1}(f_0)<\hat{\text{T}}_{f_0}=\hat{\textbf{T}}(f_0)$. Then, there exists $\textbf{k}_1$ with $\textbf{k}_1(f_0)>\boldsymbol{\tau}_1(f_0)$ and $\textbf{k}_1(f_0)=\boldsymbol{\tau}_1(f_0)$ for any $f\neq f_0$ such that ${\bf{T}}_{\textbf{k}_1}(f_0)=\hat{T}_{f_0}$ and ${\bf{T}}_{\textbf{k}_1}(f)\le \hat{\bf{T}}(f)$ if $f\neq f_0$. That is, $\textbf{k}_1\in \mathcal{K}_1$, which is a contradiction to the definition of ${\bf{T}}_{\boldsymbol{\tau}_1}$. Thus, $\mathcal{D}_{\boldsymbol{\tau}_1}$ is an ideal spherical pattern with $\textbf{T}_{\boldsymbol{\tau}_1}=\hat{{\bf{T}}}$.

\textbf{Step IV.} We prove the uniqueness of $\mathcal{D}_{\textbf{k}}$ with $\textbf{T}_{\textbf{k}}=\hat{\textbf{T}}$.

\begin{figure}[H]
{
\centering
\includegraphics[height=6cm]{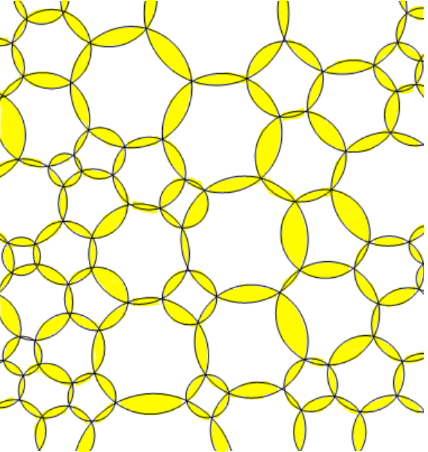} 
\caption{The bigons of the spherical ideal circle pattern.}
\label{Fig12}

}
\end{figure}

Assume that there is another ideal spherical circle pattern $\mathcal{D}_{\boldsymbol{\tau}^{\prime}}$ such that $\textbf{T}_{\boldsymbol{\tau}^{\prime}}=\hat{\textbf{T}}$. Clearly, $\boldsymbol{\tau}^{\prime}\in \mathcal{K}_1$. Thus, $\boldsymbol{\tau}^{\prime}(f)\le \boldsymbol{\tau}_1(f)$ for any $f\in F$. We need to show that $\boldsymbol{\tau}^{\prime}(f)= \boldsymbol{\tau}_1(f)$ for any $f\in F$. Suppose this is not true. Then 
there exists $f_0\in F$ such that $\boldsymbol{\tau}^{\prime}(f_0)<\boldsymbol{\tau}_1(f_0)$. We derive a contradiction by using an area comparison. 

From the Gauss-Bonnet formula, the area of a bigon $\Omega_e$ for a spherical ideal circle pattern $\mathcal{D}_\textbf{k}$ is expressed as 
\begin{equation}\label{equ-12}
    Area(\Omega_e)=2\theta_e-T_{f,e}-T_{g,e},
\end{equation}
where $f,g\in F$ are two faces sharing an edge $e$.
Each $T_{f,e}$, $e\in E(f)$, contributes once and only once to the total geodesic curvature ${\bf{T}}_{\textbf{k}}(f)$ of the conical circle corresponding to a face $f$. Let $\Omega$ be the union of all the bigons for the circle pattern $\mathcal{D}_\textbf{k}$. Thus,
\begin{equation}\label{equ-13}
Area(\Omega)=\sum_{e\in E}Area(\Omega_e)=2\sum_{e\in E}\theta_e-\sum_{f\in F}{\bf{T}}_{\textbf{k}}(f).
\end{equation}
Since $\mathcal{D}_{\boldsymbol{\tau}_1}$ and $\mathcal{D}_{\boldsymbol{\tau}^{\prime}}$ are ideal circle patterns with ${\bf{T}}_{\boldsymbol{\tau}_1}(f)={\bf{T}}_{\boldsymbol{\tau}^{\prime}}(f)=\hat{T}_f$ for any $f\in F$, it follows that
\begin{equation}\label{equ-14}
Area_{\boldsymbol{\tau}_1}(\Omega)=2\sum_{e\in E}\theta_e-\sum_{f\in F}\hat{T}_{f}=Area_{\boldsymbol{\tau}^{\prime}}(\Omega).
\end{equation}
Since $\boldsymbol{\tau}^{\prime}(f_0)<\boldsymbol{\tau}_1(f_0)$ for some $f_0\in F$ and $\boldsymbol{\tau}^{\prime}(f)\le \boldsymbol{\tau}_1(f)$ for any other $f\in F$, it follows from (d) of Lemma \ref{lem-2.2} that 
$Area_{{\boldsymbol{\tau}}^{\prime}}(\Omega)<Area_{\boldsymbol{\tau}_1}(\Omega)$. This contradicts the equation (\ref{equ-14}). Therefore,  $\boldsymbol{\tau}_1=\boldsymbol{\tau}^{\prime}$, which proves the uniqueness of the ideal spherical circle pattern $D_{\textbf{k}}$ with $\textbf{T}_{\textbf{k}}=\hat{\textbf{T}}$.
\end{proof}

In fact, one may provide an alternative proof of Theorem \ref{Rigidity theorem}
by using the superpattern set $\mathcal{K}_2$ of $\hat{{\bf{T}}}$ defined as 
\begin{equation}\label{equ-7}
\mathcal{K}_2=\left\{\begin{array}{l|l}\textbf{k}:F\rightarrow(0,+\infty)\end{array}\begin{array}{l}
\mathcal{D}_{\textbf{k}}\text{ is a spherical ideal circle}\\ \text{pattern underlining }(S,G,\Theta)\\\text{ with } {\bf{T}}_{\textbf{k}}(f)\geq\hat{{\bf{T}}}(f),\forall f\in F\end{array}\right \}.
\end{equation}
However, it takes more effort to prove that $\mathcal{K}_2$ is not empty. We outline a proof as follows. 

We start with $\textbf{k}_1=\langle k_{11},0,0,\cdots,0\rangle$ satisfying ${\bf{T}}_{\textbf{k}_1}(f_1)=\hat{T}_1.$ 

Let $\textbf{k}_2=\langle k_{21},k_{22},0,\cdots,0\rangle\text{, where }k_{21}=k_{11}+k_{22}$, such that ${\bf{T}}_{\textbf{k}_2}(f_2)=\hat{T}_2.$ Since $\frac{\partial  T_{i,e}}{\partial  k_i}+\frac{\partial  T_{i,e}}{\partial  k_j}>0$, ${\bf{T}}_{\textbf{k}_2}(f_1)>\hat{T}_1$.

Let $\textbf{k}_3=\langle k_{31},k_{32},k_{33},0,\cdots,0\rangle\text{, where }k_{31}=k_{21}+k_{33}\text{ and }k_{32}=k_{22}+k_{33}$, such that ${\bf{T}}_{\textbf{k}_3}(f_1)=\hat{T}_1.$ Since $\frac{\partial  T_{i,e}}{\partial  k_i}+\frac{\partial  T_{i,e}}{\partial  k_j}>0$, we obtain ${\bf{T}}_{\textbf{k}_3}(f_2)>\hat{T}_2$ and ${\bf{T}}_{\textbf{k}_3}(f_3)>\hat{T}_3.$

Inductively, we obtain $\textbf{k}_{|F|}=\langle k_{|F|1},k_{|F|2},\cdots,k_{|F||F|}\rangle$ satisfying $$k_{|F|j}=k_{(|F|-1)j}+k_{|F||F|} \;\text{ for any }1\leq j\leq |F|-1$$ and ${\bf{T}}_{\textbf{k}_{|F|}}(f_{|F|})=\hat{T}_{|F|}.$ Since $\frac{\partial  T_{i,e}}{\partial  k_i}+\frac{\partial  T_{i,e}}{\partial  k_j}>0$, it follows that $${\bf{T}}_{\textbf{k}_{|F|}}(f_j)>\hat{T}_j\;\text{ for any } 1\leq j\leq |F|-1.$$ Thus, $\textbf{k}_{|F|}\in \mathcal{K}_2$, which means that $\mathcal{K}_2$ is not empty. 

Similar to $\mathcal{K}_1$, $\mathcal{K}_2$ has the net property. Define $\inf \mathcal{K}_2$ to be the map $\boldsymbol{\tau}_2: F\rightarrow (0, \infty)$ by setting 
\begin{equation}\label{equ-9}
\boldsymbol{\tau}_2(f)=\inf_{\textbf{k}^{\prime}\in \mathcal{K}_2}\textbf{k}^{\prime}(f),\forall f_i\in F.
\end{equation}
We can prove that $\mathcal{D}_{\boldsymbol{\tau_2}}$ is the unique ideal spherical circle pattern $\mathcal{D}_\textbf{k}$ with $\textbf{T}_\textbf{k}=\hat{\textbf{T}}$. It follows that $\boldsymbol{\tau_2}=\boldsymbol{\tau_1}$.

\section{Convergence of Thurston's algorithm}\label{sec-4}
In the proof of Theorem \ref{Rigidity theorem}, we have explained that for any $\textbf{k}^0:F\rightarrow (0,+\infty)$, there is a unique ideal spherical circle pattern $\mathcal{D}^0$ (up to isometry) underlying $(S, G,\Theta)$ such that $k_i^0$ is the geodesic curvature of the boundary circle of the disk $D_i$ corresponding to the face $f_i$ for each $1\leq i \leq |F|$, and, furthermore, the corresponding ${\bf{T}}^{0}\in \mathcal{T}$, where ${\bf{T}}^{0}:F\rightarrow (0,+\infty)$ represents the vector of the total geodesic curvatures of the boundary circles of the disks in the circle pattern $\mathcal{D}^0$. This secures an initial circle pattern for the Thurston algorithm to start. Now we prove the convergence of this algorithm. 

\begin{proof}[Proof of Theorem \ref{Thurston's algorithm}]
Given a vector $\textbf{k}=\langle k_1,k_2,\cdots,k_{|F|}\rangle\in (0, +\infty)^{|F|}$, let $\textbf{u}=\langle \ln k_1, \ln k_2, \cdots, \ln k_{|F|}\rangle$. We briefly write it as $\textbf{u}=\ln \textbf{k}$ or $\textbf{k}=e^{\textbf{u}}$. Then $\textbf{u}\in (-\infty, +\infty)^{|F|}$. We define a total geodesic curvature adjusting map $\Psi:\mathbb{R}^{|F|}\rightarrow\mathbb{R}^{|F|}$ according to the process given in Section \ref{sec-1.5} to obtain the $1^{st}$ approximation $\mathcal{D}^1$. Given an initial vector $\textbf{k}^0=\langle k_1^0,k_2^0,\cdots,k_{|F|}^0\rangle\in (0,+\infty)^{|F|}$, define
$$\Psi(\ln\textbf{k}^0)=\ln\textbf{k}^1=\langle \ln k_1^1,\ln k_2^1,\cdots,\ln k_{|F|}^1\rangle.$$ In brief, we express this map as $\textbf{u}^1=\Psi(\textbf{u}^0)$. 

For each face $f_i\in F$, define $\Psi_i$ as the map that only changes the geodesic curvature $\textbf{k}_i$ of the boundary circle $C_i$ of the disk corresponding to $f_i$ to $\textbf{k}_i'$ such that the total geodesic curvature of the boundary circle $C_i'$ of the disk corresponding to $f_i$ is equal to $\hat{\textbf{T}}(f_i)=\hat{T}_i$. That is, 
$$\Psi_i(\textbf{u})=\Psi_i(\ln \textbf{k})=(\ln k_1, \cdots, \ln k_{i-1}, \ln k_i', \ln k_{i+1}, \cdots, \ln k_{|F|}).$$ Denote by $\psi_i(\textbf{u})=\ln k_i'$. Then
$$\Psi_i(\textbf{u})=(u_1, \cdots, u_{i-1}, \psi_i(\textbf{u}), u_{i+1}, \cdots , u_{|F|}).$$
The iterated map $\Psi$ under Thurston's algorithm is given by
\begin{equation}\label{Iterated map for Thurston algorithm}
\begin{array}{cccl}
\Psi: & (-\infty, +\infty)^{|F|} & \rightarrow & (-\infty, +\infty)^{|F|}\\& \textbf{u} & \mapsto & \Psi(\textbf{u})=(\psi_1(\textbf{u}), \psi_2(\textbf{u}), \cdots, \psi_{|F|}(\textbf{u})).
\end{array}
\end{equation}
Note that $\hat{\textbf{u}}=\ln \hat{\textbf{k}}$ is a fixed point of each $\Psi_i$ and hence a fixed point of $\Psi$. 

Define $B(\hat{\textbf{u}},\rho)$ to be the closed box centered at $\hat{\textbf{u}}$ and of size $\rho >0$; that is,
$$B(\hat{\textbf{u}},\rho)=\{\textbf{u}\in (-\infty,+\infty)^{|F|}: ||\textbf{u}-\hat{\textbf{u}}||_{\infty }=\max_{1\le j\le |F|}|u_j-\hat{u}_j|\le \rho\}.$$
We show that there exists $0<\lambda <1$, depending on $B(\hat{\textbf{u}}, \rho)$, such that for each $1\le i\le |F|$ and for any two points $\textbf{u}, \textbf{u}'\in B(\hat{\textbf{u}}, \rho)$, 
\begin{equation}\label{Local contraction}
|\psi_i(\textbf{u})-\psi_i(\textbf{u}')|\le \lambda ||\textbf{u}-\textbf{u}'||_{\infty }.
\end{equation}

Given any $\textbf{k}=\langle k_1,k_2,\cdots,k_{|F|}\rangle\in (0, +\infty)^{|F|}$, let $\mathcal{D}_{\textbf{k}}$ be the unique ideal spherical circle pattern (up to isometry) underlying $(S, G,\Theta)$ such that $k_{f_i}=\textbf{k}(f_i)$ is the geodesic curvature of the boundary circle $C_i$ of the disk corresponding to the face $f_i\in F$. The total geodesic curvature of $C_i$ is expressed as $T_i(\textbf{k})=T_i(k_1, k_2, \cdots, k_{|F|}),\;i=1,\cdots,|F|$. From the proof of Lemma \ref{lem-2.3}, we can see that $T_i(\textbf{k})$ is a differentiable function of $\textbf{k}$. Thus, $T_i(e^{\textbf{u}})$ is a differentiable function of $\textbf{u}$, which we continue to denote by $T_i(\textbf{u})$.

For each face \( f_i\in F \), we denote by \( j \sim i \) if \( f_j \) is a neighboring face of \( f_i \). From (a) and (b) of Lemma \ref{lem-2.2}, we know that the total geodesic curvature \( T_i(\mathbf{u})=T_i(e^{\textbf{u}}) \) has the following properties:
\begin{enumerate}

   \item  \( T_i(\mathbf{u}) \) is strictly increasing with respect to \( u_i \);  

   \item \( T_i(\mathbf{u}) \) is strictly decreasing with respect to each neighboring curvature \( u_j \) (\( j \sim i \));

   \item \( T_i(\mathbf{u}) \) depends only on \( u_i \) and \( \{u_j\}_{j \sim i} \), and is independent of all other curvatures \( u_l \) with \( l \nsim i \).

\end{enumerate}
Applying the implicit differentiation to $T_{i}(\Psi_i(\textbf{u})) = \hat{T}_{i}$, we obtain
\begin{equation} \label{implicit differentiaion}
\frac{\partial \psi_i(\textbf{u})}{\partial u_j} =-\frac{\frac{\partial T_{i}}{\partial u_j}(\Psi _i(\textbf{u}))}{\frac{\partial T_{i}}{\partial u_i}(\Psi_i(\textbf{u}))}, \quad j\neq i,\quad\text{and} \quad\frac{\partial \psi_i(\textbf{u})}{\partial u_i}=0.
\end{equation}
Given $f_i\in F$, the Gauss-Bonnet formula in (\ref{equ-12}) implies that 
$$
\sum_{e\in E(f_i)}Area(\Omega_e)=2\sum_{e\in E(f_i)}\theta_e-T_i-\sum_{j\sim i}T_{j,e}.
$$
By taking the partial derivatives of both sides of the above equation with respect to $u_i$, we obtain
$$
\begin{aligned}
-\frac{\partial\sum_{e\in E(f_i)} Area(\Omega_e)}{\partial u_i}&=\frac{\partial T_i}{\partial u_i}+\sum_{j\sim i}\frac{\partial T_{j,e}}{\partial u_i}
=\frac{\partial T_i}{\partial u_i}+\sum_{j\sim i}\frac{\partial T_{i,e}}{\partial u_j}\\ &=\frac{\partial T_i}{\partial u_i}+\sum_{j\sim i}\frac{\partial T_{i}}{\partial u_j}
=\frac{\partial T_i}{\partial u_i}+\sum_{j\neq i}\frac{\partial T_{i}}{\partial u_j}\\
\end{aligned}
$$ since $\frac{\partial T_{j,e}}{\partial u_i}=0$ if $f_i$ and $f_j$ don't share the edge $e$, and $\frac{\partial T_{i}}{\partial u_j}=0$ if $j\neq i$ and $j\nsim i$.

Since each $\Psi _i$ is continuous on the closed box $B(\hat{\textbf{u}},\rho)$, there exists a large enough $\rho'>\rho$ such that $\Psi_i(B(\hat{\textbf{u}},\rho))\subset B(\hat{\textbf{u}},\rho')$ for all $1\le i\le |F|$. Applying (d) of Lemma \ref{lem-2.2}, we know that there is a positive real number $\epsilon$, depending on $B(\hat{\textbf{u}},\rho')$, such that for any $\textbf{u}\in B(\hat{\textbf{u}},\rho')$,
$$\frac{\partial T_i}{\partial u_i}+\sum_{j\neq i}\frac{\partial T_i}{\partial u_j}=-\frac{\partial\sum_{e\in E(f_i)} Area(\Omega_e)}{\partial u_i}\ge \epsilon,\quad 1\leq i\leq |F|.$$
In the meantime, there is $M>0$, which also depends on $B(\hat{\textbf{u}},\rho')$, such that
$$
\left|\frac{\partial T_{i}}{\partial u_i}\right|\le M, \; \forall \;\textbf{u}\in B(\hat{\textbf{u}},\rho') \text{ and }\forall \;1\le i\le |F|.
$$
By (a) and (b) of Lemma \ref{lem-2.2}, we know that $\frac{\partial T_{i}}{\partial u_i}>0$ and $\frac{\partial T_i}{\partial u_j}\le 0$ for any $j\neq i$.
It follows that 
$$
\lambda 
=\sup_{\textbf{u}\in B(\hat{\textbf{u}},\rho')}\left\{\frac{\sum_{j\neq i}\left|\frac{\partial T_{i}}{\partial u_j}\right|}{\left|\frac{\partial T_{i}}{\partial u_i}\right|}\right\}
\le \sup_{\textbf{u}\in B(\hat{\textbf{u}},\rho')}\left\{\frac{\left|\frac{\partial T_{i}}{\partial u_i}\right|-\epsilon}{\left|\frac{\partial T_{i}}{\partial u_i}\right|}\right\}\\
\le 1-\frac{\epsilon}{M}
< 1.
$$
Using the expressions in (\ref{implicit differentiaion}), we obtain 
\begin{equation}\label{contraction constant}\sup_{\textbf{u}\in  B(\hat{\textbf{u}},\rho)} \sum _{j}|\frac{\partial \psi_i}{\partial u_j}(\textbf{u})|\le \lambda <1.
\end{equation}

Given $\textbf{u},\textbf{u}^{\prime}\in B(\hat{\textbf{u}},\rho)$, it follows from the mean value theorem for a function of several variables that
\begin{equation}\label{difference of phi_i}
\psi_i(\textbf{u})-\psi_i\left(\textbf{u}^{\prime}\right)=\int_0^1 \nabla \psi_i\left(\textbf{u}+t\left(\textbf{u}-\textbf{u}^{\prime}\right)\right) \cdot\left(\textbf{u}-\textbf{u}^{\prime}\right) d t. 
\end{equation}
Thus,
\begin{equation}\label{Esitimate on the difference of phi_i}
\begin{aligned}
 \left|\psi_i(\textbf{u})-\psi_i\left(\textbf{u}^{\prime}\right)\right| \leq &\int_0^1\left|\nabla \psi_i\left(\textbf{u}+t\left(\textbf{u}-\textbf{u}^{\prime}\right)\right)\left(\textbf{u}-\textbf{u}^{\prime}\right)\right| d t \\
\leq & \int_{0}^1\left( \sum_{j}\left|\frac{\partial \psi_i}{\partial u_j}\left( \textbf{u}+t\left(\textbf{u}-\textbf{u}^{\prime}\right)\right)\right| \cdot\left|u_j-u_j^{\prime}\right| \right)d t\\
= & (\sup _{0\leq t\leq 1} \{\sum_{j}|\frac{\partial \psi_i}{\partial u_j}(\textbf{u}+t\left(\textbf{u}-\textbf{u}^{\prime} ))\right| \})  \cdot\left\|\textbf{u}-\textbf{u}^{\prime}\right\|_{\infty}\\<&\lambda \left\|\textbf{u}-\textbf{u}^{\prime}\right\|_{\infty}.
\end{aligned}
\end{equation} 
Since 
$\left|\psi_i(\textbf{u})-\psi_i\left(\textbf{u}^{\prime}\right)\right| \leq \lambda \left\|\textbf{u}-\textbf{u}^{\prime}\right\|_{\infty}$ holds for each $1\le i\le |F|$, it follows that $$||\Psi(\textbf{u})-\Psi(\textbf{u}')||_{\infty }<\lambda ||\textbf{u}-\textbf{u}'||_{\infty }.$$ Hence $\Psi$ maps $B(\hat{\textbf{u}}$ into itself and $\Psi: B(\hat{\textbf{u}},\rho)\rightarrow B(\hat{\textbf{u}},\rho)$ is a contraction. Therefore, for each $\textbf{u}\in B(\hat{\textbf{u}},\rho)$, $\Psi^m(\textbf{u})$ converges to the unique fixed point $\hat{\textbf{u}}$ as $m\rightarrow +\infty$.
We complete a proof for the convergence of the Thurston algorithm. 

It is easy to see that $\Psi$ maps the set $\mathcal{K}_1$ (defined by (\ref{equ-6})) into itself. Thus, for each $1\le i \le |F|$, $(\Psi^m(\textbf{u}))_i$ increases to $\hat{u}_i$ as $m\rightarrow +\infty$.
The map $\Psi$ preserves the set $\mathcal{K}_2$ (defined by (\ref{equ-7})) also, which implies that for each $1\le i \le |F|$, $(\Psi^m(\textbf{u}))_j$ decreases to $\hat{u}_i$ as $m\rightarrow +\infty$.  
\end{proof}

\bigskip
{\bf Acknowledgments:} Jun Hu is supported by a fellowship leave award of City University of New York for the adademic year 2025-26. Yi Qi is supported by NSFC (Grant Number 12271017). Yu Sun is supported by NSFC (Grant Number 12501097) and a university level natural science
foundation of Nanjing Institute of Technology (Grant number  3534113223051). The authors thank Guangming Hu for some helpful discussions.

\bibliographystyle{amsplain}
\bibliography{refs.bib}

\providecommand{\bysame}{\leavevmode\hbox to3em{\hrulefill}\thinspace}
\providecommand{\MR}{\relax\ifhmode\unskip\space\fi MR }
\providecommand{\MRhref}[2]{%
  \href{http://www.ams.org/mathscinet-getitem?mr=#1}{#2}
}
\providecommand{\href}[2]{#2}
\begin{thebibliography}{1}

\bibitem{MR2022715}
Alexander~I. Bobenko and Boris~A. Springborn, \emph{Variational principles for circle patterns and {K}oebe's theorem}, Trans. Amer. Math. Soc. \textbf{356} (2004), no.~2, 659--689. \MR{2022715}

\bibitem{MR1106755}
Yves Colin~de Verdi\`ere, \emph{Un principe variationnel pour les empilements de cercles}, Invent. Math. \textbf{104} (1991), no.~3, 655--669. \MR{1106755}

\bibitem{MR4683863}
Xin Nie, \emph{On circle patterns and spherical conical metrics}, Proc. Amer. Math. Soc. \textbf{152} (2024), no.~2, 843--853. \MR{4683863}

\bibitem{Thurston1979TheGA}
William~P. Thurston, \emph{The geometry and topology of 3-manifolds}, Princeton Lecture Notes, 1976.

\bibitem{Troyanov1991}
Marc Troyanov, \emph{Prescribing curvature on compact surfaces with conical singularities}, Trans. Amer. Math. Soc. \textbf{324} (1991), no.~2, 793--821. \MR{1005085}

\end{thebibliography}

\newpage
\begin{center}
\rule{\textwidth}{0.4pt}\\[0.3cm]
{\large \textbf{Affiliation and Contact Information}}\\[0.5cm]
\end{center}

\noindent
\textbf{Lishan Li} \\
School of Mathematical Sciences, Beihang University \\
Beijing, 102206, P. R. China \\
Email: \href{mailto: lishan-li@buaa.edu.cn}{lishan-li@buaa.edu.cn}

\vspace{0.5cm}

\noindent
\textbf{Jun Hu} \\
Department of Mathematics, Brooklyn College of CUNY\\
Brooklyn, NY 11210, USA\\
Ph.D. Program in Mathematics, Graduate Center of CUNY\\
365 Fifth Avenue, New York, NY 10016, USA \\
Email: \href{mailto:junhu@Brooklyn.cuny.edu}{junhu@Brooklyn.cuny.edu} or \href{mailto:JHu1@gc.cuny.edu}{JHu1@gc.cuny.edu}

\vspace{0.5cm}

\noindent
\textbf{Yi Qi} \\
School of Mathematical Sciences, Beihang University \\
Beijing, 102206, P. R. China \\
Email: \href{mailto:yiqi@buaa.edu.cn}{yiqi@buaa.edu.cn}

\vspace{0.5cm}

\noindent
\textbf{Yu Sun} \\
School of Mathematics and Physics, Nanjing institute of technology\\
Nanjing, 211100, P.R. China\\
Email: \href{mailto:yusun15185105160@163.com}{yusun15185105160@163.com}

\end{document}